\documentclass[10pt]{amsart}
\usepackage{amssymb,amsmath,geometry,latexsym,
graphics,tabularx,shapepar,enumerate}

\newcommand{\ZZ}{\mathbb{Z}}
\newcommand{\FF}{\mathbb{F}}
\newcommand{\CC}{\mathbb{C}}
\newcommand{\QQ}{\mathbb{Q}}
\newcommand{\RR}{\mathbb{R}}

\newcommand{\NN}{\mathbb{N}}

\newcommand{\TT}{\mathbb{T}}

\usepackage[OT2,T1]{fontenc}
\DeclareSymbolFont{cyrletters}{OT2}{wncyr}{m}{n}
\DeclareMathSymbol{\Sha}{\mathalpha}{cyrletters}{"58}

\newtheorem{Theorem}{Theorem}
\newtheorem{Lemma}[Theorem]{Lemma}
\newtheorem{Proposition}[Theorem]{Proposition}

\newtheorem{Corollary}[Theorem]{Corollary}
\newtheorem{Definition}[Theorem]{Definition}
\newtheorem{Remark}[Theorem]{Remark}

\date{January, 2016}

\title[Multiple zeta values]{A note on multiple zeta values in Tate algebras}

\author{F. Pellarin}
\address{Federico Pellarin: Institut Camille Jordan, UMR 5208 Site de Saint-Etienne, 23 rue du Dr. P. Michelon, 42023 Saint-Etienne,
France}
\email{federico.pellarin@univ-st-etienne.fr}



\begin{document}

\maketitle

\begin{abstract}

In this note, we shall discuss a generalization of Thakur's multiple zeta values
and allied objects, in the framework of function fields of positive characteristic and more precisely,
of periods in Tate algebras.

\end{abstract}

\section{Introduction}
Let $A=\FF_q[\theta]$ be the ring of polynomials in an indeterminate $\theta$  with coefficients in $\FF_q$ the finite field with 
$q$ elements and characteristic $p$, let $K$ be the fraction field of $A$ and 
$K_\infty$ the completion of $K$ at the infinity place $\infty$.  
For $d\geq 0$ an integer, we denote by $A^+(d)$ the set of monic polynomials of $A$
of degree $d$. Carlitz studied, in \cite{CAR}, the so-called {\em Carlitz zeta values}:
$$\zeta_C(n):=\sum_{a\in A^+}a^{-n}\in K_\infty,\quad n\geq 1.$$
It is likely that the formal analogy of these objects with the {\em classical zeta values} $$\zeta(n)=\sum_{i\geq 1}i^{-n}\in\RR$$ with $n$ 
integer (convergence occurs only if $n\geq 2$) was the main motivation for his study. 
In a more modern approach, we can say that Carlitz suggested, with his first pioneering papers, to develop an arithmetic theory 
of {\em periods} over the ring $\FF_q[\theta]$ in parallel with the study of the arithmetic
theory of periods over $\ZZ$. 

In all the following, if $R$ is a ring, $R^\times$ denotes the group of 
the multiplicative invertible elements of $R$.
It was proved by Carlitz in \cite{CAR} that, if $n\equiv0\pmod{q-1}$, 
\begin{equation}\label{carlitzresultpi}\zeta_C(n)\in K^\times\widetilde{\pi}^n,\end{equation}
where $\widetilde{\pi}$ is the value in $\CC_\infty=\widehat{K_\infty^{ac}}$ (\footnote{This is the completion of 
an algebraic closure of $K_\infty$.}) of a convergent infinite product
\begin{equation}\label{pi}
\widetilde{\pi}:=-(-\theta)^{\frac{q}{q-1}}\prod_{i=1}^\infty(1-\theta^{1-q^i})^{-1}\in (-\theta)^{\frac{1}{q-1}}K_\infty,
\end{equation}
uniquely defined up to the multiplication by an element of $\FF_q^\times=\FF_q\setminus\{0\}$
(corresponding to the choice of a root $(-\theta)^{\frac{1}{q-1}}$). It has been proved in a variety of
ways (see \cite{PEL0} to see the most relevant ones) that $\widetilde{\pi}$ is moreover transcendental over
$K$.

The element $\widetilde{\pi}$ is a fundamental period of the {\em Carlitz exponential} $\exp_C$  (Goss, \cite[\S 3.2]{GOS}), that is, the unique  
surjective, entire, $\FF_q$-linear  function $$\exp_C:\CC_\infty\rightarrow\CC_\infty$$ of kernel $\widetilde{\pi}\FF_q[\theta]$ such that its first derivative satisfies $\exp_C'=1$ (note that, since we are in a
characteristic $p>0$ environment, a function with constant derivative is not necessarily $\CC_\infty$-linear).

In his book \cite[\S 5.10]{THA}, Thakur also consider several variants of classical multiple zeta values in the 
context of the Carlitzian arithmetic over the ring $A$. We mention here what we think is the 
most relevant.
For $n_1,\ldots,n_r\in\ZZ_{\geq 1}$, Thakur defines, as one of the analogues
of the classical multiple zeta values in the Carlitzian setting:
\begin{equation}\label{defmultizeta}
\zeta_C(n_1,\ldots,n_r)=\sum_{a_i\in A^+ \atop |a_1|>\cdots>|a_r|}\frac{1}{a_1^{n_1}\cdots a_r^{n_r}}\in K_\infty.
\end{equation} Here, for $x\in\CC_\infty^\times$, we write $|x|=q^{-v_\infty(x)}$
where $v_\infty$ is the valuation of $\CC_\infty$ (so that $v_\infty(\theta)=-1$) and we define $|0|:=0$. If $r=0$ we further set the 
corresponding Thakur multiple zeta value $\zeta_C(\emptyset)$ to be equal to $1$.

Classically, one of the reasons we could get interested
in multiple zeta values is the need of "enveloping" zeta values in the "simplest" $\QQ$-algebra
possible. From Euler, it is well known that the zeta values $\zeta(2),\zeta(4),\ldots$ all belong to the
$\QQ$-algebra $\QQ[\zeta(2)]$. In general, the other zeta values 
are not expected to belong to this algebra. However, they belong to the $\QQ$-algebra $Z_\RR\subset\RR$ generated by the multiple zeta values. It is known that the product of two multiple zeta values is a
$\QQ$-linear combination of multiple zeta values, and this algebra also has a more natural structure. 

We expect that $Z_\RR$
is isomorphic to the algebra $\QQ\langle f_3,f_5,\ldots\rangle_\Sha\otimes_\QQ\QQ[\zeta(2)]$,
where $\QQ\langle f_3,f_5,\ldots\rangle_\Sha$ is the $\QQ$-algebra generated
by the non-commutative words in the alphabet with letters $f_3,f_5\ldots$ with, as a product, the
shuffle product $\small{\Sha}$ (see Brown's \cite{BRO}). A folklore conjecture comes in support of this guess;
the number $\pi$ and the zeta values $\zeta(3),\zeta(5),\ldots$ are expected to be 
algebraically independent over $\QQ$. Multiple zeta values are thus expected to 
provide a natural basis of this $\QQ$-algebra. See also \cite{KAN}
for the definition of a $\QQ$-algebra of {\em finite multi-zeta values} which could 
offer a nice realization of the algebra $\QQ\langle f_3,f_5,\ldots\rangle_\Sha$. 

Similarly, in the Carlitzian setting we note that, after (\ref{carlitzresultpi}), 
the values $\zeta_C(n)$ with $n>0$ divisible by $q-1$ are all contained in the 
$K$-algebra $K[\zeta_C(q-1)]$, which is isomorphic to $K[X]$ for an indeterminate $X$.
However, the remaining Carlitz zeta values $\zeta_C(1),\ldots$ do not belong 
to this algebra (if $q>2$). Indeed, Chang and Yu proved in \cite{CHA&YU} that
$\widetilde{\pi}$ and the Carlitz zeta values $\zeta_C(n)$ with $n\geq 1$, $q-1\nmid n$
and $p \nmid n$ with $p$ the prime number dividing $q$ are algebraically independent
(we recall that these authors, in ibid., use the powerful algebraic independence 
methods introduced by Papanikolas in \cite{PAP}); see also \cite{CHA1,CHA2,CHA&PAP&YU}.

Just as for the algebra $Z_\RR$, Thakur proved in \cite{THA2} that the 
product of two multiple zeta values as in (\ref{defmultizeta}) is a linear combination
(this time with coefficients in $\FF_p$) of such multiple zeta values. Thakur also mentioned to 
the author of the present note that G. Todd's numerical computations have led to a good 
understanding of (conjectural) relations among Thakur's multiple zeta values; the relations are universal in a sense that is described in ibid. 
Compared to the classical setting, the difficulty here is to handle the product of
the so-called {\em power sums} (see later).
We denote by $Z_{K_\infty}$
the $K$-sub-algebra of $K_\infty$ generated by the multiple zeta values (\ref{defmultizeta});
even conjecturally, in spite of the striking results of algebraic independence
mentioned above, we know very little about the structure of this algebra.
In particular, we presently do not know what could be the analogue structure 
which could play the role of the algebra $\QQ\langle f_3,f_5,\ldots\rangle_\Sha$
in this setting.

In this note, we shall discuss of a generalization of the Thakur multiple zeta values which,
 so far, has no counterpart in the classical setting. For this purpose, we 
note that  $A$ is an algebra over $\FF_q$
($\ZZ$ is not an algebra over a field). Therefore, a series of advantages occurs 
in the Carlitzian framework, notably the possibility to use the tensor product over $\FF_q$.
We consider variables $t_1,\ldots,t_s$ over $K$ and we write $\underline{t}_s$
for the family of variables $(t_1,\ldots,t_s)$. We denote by $\boldsymbol{F}_s$ the field $\FF_q(\underline{t}_s)$,
so that $\boldsymbol{F}_0=\FF_q$.

In all the following, if $R$ is a ring, we denote by $R^*$ the underlying multiplicative monoid
(inclusive of the element $0$).
Note that $A^+$ is a multiplicative sub-monoid of $A^*$. We denote by $\FF_q^{ac}$
the algebraic closure of $\FF_q$ in $\CC_\infty$.

\begin{Definition}
{\em A monoid homomorphism $\sigma:A^+\rightarrow(\FF_q^{ac}\otimes_{\FF_q}\boldsymbol{F}_s)^*$ is called a
{\em semi-character}. The {\em trivial semi-character} 
is the map $\boldsymbol{1}:A^+\rightarrow\{1\}$. 
Let $\sigma$ be a semi-character. We say that it is {\em of Dirichlet type}
if there exist $\FF_q$-algebra homomorphisms
$$\rho_i:A\rightarrow\FF_q^{ac}\otimes_{\FF_q}\boldsymbol{F}_s,\quad i=1,\ldots,s,$$
such that $\sigma(a)=\rho_1(a)\cdots\rho_s(a)$ for all $a\in A^+$. The integer $s$ is called the
{\em length}.
By convention, the semi-character $\boldsymbol{1}$ is the unique semi-character of 
Dirichlet type of length $0$.}\end{Definition}
For example, setting $t=t_1$, the map $\chi_t:A^+\rightarrow\FF_q[t]^*\subset(\FF_q^{ac}\otimes_{\FF_q}\boldsymbol{F}_s)^*$ defined by $\chi_t(a)=a(t)$
(\footnote{The "evaluation at $\theta=t$", in other words, the map 
which sends a polynomial $a=a(\theta)=a_0+a_1\theta+\cdots+a_r\theta^r$ with the coefficients $a_0,\ldots,a_r\in\FF_q$
to the polynomial $a(t)=a_0+a_1t+\cdots+a_rt^r\in\FF_q[\theta]$.}) is a semi-character of Dirichlet 
type.
Let $\zeta$ be an element of $\FF_q^{ac}$. The map $a\in A^+\mapsto \chi_{\zeta}(a)=a(\zeta)\in\FF_q^{ac}
\subset(\FF_q^{ac}\otimes_{\FF_q}\boldsymbol{F}_s)^*$
is also a semi-character of Dirichlet type, and the same can be said if 
we now pick elements $\zeta_1,\ldots,\zeta_s\in\FF_q^{ac}$ and consider the map
$\chi_{\underline{\zeta}}:a\mapsto\chi_{\zeta_1}(a)\cdots\chi_{\zeta_s}(a)$ (this is more commonly called
a ``Dirichlet character").
The map $a\in A^+\mapsto \FF_q[t]^*$
which sends $a$ to $t^{\deg_\theta(a)}$ is a semi-character, but it can be proved that 
it is not of Dirichlet type.

\begin{Definition}\label{defpowersums}{\em 
Let $\sigma:A^+\rightarrow(\FF_q^{ac}\otimes_{\FF_q}\boldsymbol{F}_s)^*$ be a semi-character.
The associated {\em twisted power sum} of order $k$ and degree $d$ is the sum:
$$S_d(k;\sigma)=\sum_{a\in A^+(d)}a^{-k}\sigma(a)\in \FF_q^{ac}\otimes_{\FF_q}K(\underline{t}_s).$$
More generally, let $\sigma_1,\ldots,\sigma_r$ be 
semi-characters, let $n_1,\ldots,n_r$ be integers, and $d$ a non-negative integer. The associated
{\em multiple twisted power sum} of degree $d$
is the sum
$$S_d\left(\begin{matrix}\sigma_1 & \sigma_2 & \cdots & \sigma_r\\
n_1 & n_2 & \cdots & n_r\end{matrix}\right)=
S_{d}(n_1;\sigma_1)\sum_{d> i_2> \cdots> i_r\geq 0}
S_{i_2}(n_2;\sigma_2)\cdots S_{i_r}(n_r;\sigma_r)\in \FF_q^{ac}\otimes_{\FF_q}K(\underline{t}_s).$$
The integer $\sum_in_i$ is called the 
{\em weight} and the integer $r$ is called its {\em depth}.}\end{Definition}
We can write in both ways $S_d(n;\sigma)=S_d\binom{\sigma}{n}$. Observe also that
$$S_d\left(\begin{matrix}\boldsymbol{1} & \boldsymbol{1} & \cdots & \boldsymbol{1}\\
n_1 & n_2 & \cdots & n_r\end{matrix}\right)=S_d(n_1,n_2,\ldots,n_r)\in K,$$
in the notations of Thakur in \cite[\S 1.2]{THA2}. We hope that all these slightly different
notations will not bother the reader.

\begin{Definition}\label{definitionmzv}{\em With $n_1,\ldots,n_r\geq 1$ and semi-characters $\sigma_1,\ldots,\sigma_r$
as above, we introduce the associated {\em multiple zeta value}
$$\zeta_C\left(\begin{matrix}\sigma_1 & \sigma_2 & \cdots & \sigma_r\\
n_1 & n_2 & \cdots & n_r\end{matrix}\right):=\sum_{d\geq 0}S_d\left(\begin{matrix}\sigma_1 & \sigma_2 & \cdots & \sigma_r\\
n_1 & n_2 & \cdots & n_r\end{matrix}\right)\in \widehat{K_\infty\otimes_{\FF_q}\FF_q^{ac}\otimes_{\FF_q}\boldsymbol{F}_s}.$$ The sum converges in the completion of the field 
$K_\infty\otimes_{\FF_q}\FF_q^{ac}\otimes_{\FF_q}\boldsymbol{F}_s$ with respect to the unique 
valuation extending the $\infty$-adic valuation of $K_\infty$ and inducing the trivial valuation over $\FF_q^{ac}\otimes_{\FF_q}\boldsymbol{F}_s$.
We will say that this is the multiple zeta value associated to the 
matrix data $$\left[\begin{matrix}\sigma_1 & \sigma_2 & \cdots & \sigma_r\\
n_1 & n_2 & \cdots & n_r\end{matrix}\right].$$ The integer $\sum_in_i$ is called the 
{\em weight} of the above multiple zeta value and the integer $r$ is called its {\em depth}.
If all the semi-characters $\sigma_1,\ldots,\sigma_r$ are of Dirichlet type, then, 
for all $1\leq i\leq r$, $\sigma_i=\rho_{i,1}\cdots\rho_{i,n_i}$ for ring homomorphisms
$\rho_{i,j}$. Then, we say that the multiple zeta value associated to the above matrix data is {\em of Dirichlet type}, and the cardinality of the set $\{\rho_{i,j};i,j\}$ is called its {\em length}.
}\end{Definition}
Again note that if $\sigma_1=\cdots=\sigma_r=\boldsymbol{1}$, then, we can write
$$\zeta_C\left(\begin{matrix}\sigma_1 & \sigma_2 & \cdots & \sigma_r\\
n_1 & n_2 & \cdots & n_r\end{matrix}\right)=\zeta_C(n_1,\ldots,n_r)\in K_\infty$$ with 
$\zeta_C(n_1,\ldots,n_r)$ as in (\ref{defmultizeta}) (of Dirichlet type, depth $r$ and length $0$). 
Further, let us assume that $r=1,n=n_1>0$ and that $\sigma=\chi_{t_1}\cdots\chi_{t_s}$.
Then, we have that
$$\zeta_C(n;\sigma)=\zeta_C\binom{\sigma}{n}=\sum_{a\in A^+}\frac{a(t_1)\cdots a(t_s)}{a^n}=\prod_P\left(1-\frac{P(t_1)\cdots P(t_s)}{P^n}\right)^{-1}\in\TT_s^\times.$$
These series have been introduced in \cite{PEL}
and extensively studied in \cite{ANG&PEL,ANG&PEL2,APTR}.
The product runs over the irreducible polynomials of $A^+$
and the convergence holds in the standard {\em $s$-dimensional Tate algebra},
which can be identified with the $\CC_\infty$-algebra of the rigid analytic functions
$B(0,1)^s\rightarrow\CC_\infty$, where $B(0,1)=\{z\in\CC_\infty;|z|\leq 1\}$. In fact,
these functions extend to entire functions $\CC_\infty^s\rightarrow\CC_\infty$
(see \cite[Corollary 8]{ANG&PEL}). More generally, if 
$\sigma_1,\ldots,\sigma_r$ are semi-characters of Dirichlet type
constructed as monomials in the ring homomorphisms $\chi_{t_1},\ldots,\chi_{t_s}$ (including the 
trivial semi-character), the multiple zeta value 
\begin{equation}\label{specimen}\zeta_C\left(\begin{matrix}\sigma_1 & \sigma_2 & \cdots & \sigma_r\\
n_1 & n_2 & \cdots & n_r\end{matrix}\right)\end{equation} belongs to $\TT_s$.

The following Proposition is easy to prove but the proof will appear elsewhere.

\begin{Proposition}\label{entireness}
With the above assumption over the semi-characters $\sigma_1,\ldots,\sigma_r$, the multiple zeta value
(\ref{specimen}),
hence of Dirichlet type and of length $\leq s$, extends to an entire function $\CC_\infty^s\rightarrow\CC_s$.
\end{Proposition}

It is presently a work in progress of the author 
to show that the product of two multiple zeta values as in Definition \ref{definitionmzv} is a linear combination, with coefficients in the field $K\otimes_{\FF_q}\FF_q^{ac}\otimes_{\FF_q}\boldsymbol{F}_s$,
of such multiple zeta values (with the various matrices of associated data not including, necessarily, the same semi-characters). We hope this will allow us to exhibit new multiple zeta values algebras
$Z_{\widehat{\FF_q^{ac}((\theta^{-1}))\otimes_{\FF_q}\boldsymbol{F}_s}}$ containing the 
algebra $Z_{K_\infty}$ and collecting the algebraic relations of $Z_{K_\infty}$ in families by specialization.

\section{Content of the present note}

Waiting for more general results, in this note we will accomplish a more modest 
objective, as we will only give a few explicit examples of shuffle products of such multiple zeta values in the 
following case: $s=2$, weight $\leq 2$, and the semi-characters $\boldsymbol{1},\sigma,\psi$ and $\sigma\psi$
of Dirichlet type,
where
$$\sigma:a\mapsto a(t_1)\in\FF_q[t_1,t_2],\quad \psi:a\mapsto a(t_2)\in\FF_q[t_1,t_2],\quad (\sigma\psi)(a)=\sigma(a)\psi(a)=a(t_1)a(t_2).$$ As an advantage of our explicit and restrictive viewpoint, we will see beautiful formulas dropping out from this 
new theory that we will apply to some new properties of the so-called ``Bernoulli-Goss" polynomials.

The matrix data we are going to handle are:
\subsubsection*{Four in weight 1}
$$\left[\begin{matrix}{c}\boldsymbol{1}\\
1\end{matrix}\right],\left[\begin{matrix}{c}\sigma\\
1\end{matrix}\right],\left[\begin{matrix}{c}\psi\\
1\end{matrix}\right],\left[\begin{matrix}{c}\sigma\psi\\
1\end{matrix}\right].$$
\subsubsection*{Four in weight 2 depth 1}
$$\left[\begin{matrix}{c}\boldsymbol{1}\\
2\end{matrix}\right],\left[\begin{matrix}{c}\sigma\\
2\end{matrix}\right],\left[\begin{matrix}{c}\psi\\
2\end{matrix}\right],\left[\begin{matrix}{c}\sigma\psi\\
2\end{matrix}\right].$$
\subsubsection*{Nine in weight 2 depth 2}
$$\left[\begin{matrix}\boldsymbol{1} & \boldsymbol{1}\\
1 & 1\end{matrix}\right],\left[\begin{matrix}\sigma & \boldsymbol{1}\\
1 & 1\end{matrix}\right],\left[\begin{matrix}\boldsymbol{1} & \sigma\\
1 & 1\end{matrix}\right],\left[\begin{matrix}\psi & \boldsymbol{1}\\
1 & 1\end{matrix}\right],\left[\begin{matrix}\boldsymbol{1} & \psi\\
1 & 1\end{matrix}\right],$$
$$\left[\begin{matrix}\sigma\psi & \boldsymbol{1}\\
1 & 1\end{matrix}\right],\left[\begin{matrix}\sigma & \psi\\
1 & 1\end{matrix}\right],\left[\begin{matrix}\psi & \sigma\\
1 & 1\end{matrix}\right],\left[\begin{matrix}\boldsymbol{1} & \sigma\psi\\
1 & 1\end{matrix}\right].$$

\medskip

We shall show the following Theorem, which provides, taking into account the above tables, a complete picture of all the 
products of two weight one multiple zeta values in the restrictive context we have 
prefixed (in two variables $t_1,t_2$, and with the semi-characters $\boldsymbol{1},\sigma,\psi$ and $\sigma\psi$), unveiling partly an extremely complex and mysterious algebra structure. From now on, we suppose that $q>2$. All the arguments presented below under this restriction can be also developed in the 
case $q=2$ with appropriate modifications, but we refrain from giving full details here.

\begin{Theorem}\label{formulas}
The following formulas hold.
\begin{enumerate}
\item $\zeta_C(1)^2=\zeta_C(2)+2\zeta_C(1,1)$,
\item $\zeta_C(1;\sigma)\zeta_C(1)=\zeta_C(2;\sigma)+\zeta_C\left(\begin{matrix} \sigma & \boldsymbol{1}\\
1 & 1\end{matrix}\right)$,
\item $\zeta_C(1;\psi)\zeta_C(1)=\zeta_C(2;\psi)+\zeta_C\left(\begin{matrix} \psi & \boldsymbol{1}\\
1 & 1\end{matrix}\right)$,
\item $\zeta_C(1;\sigma)\zeta_C(1;\psi)=\zeta_C(2;\sigma\psi)$,
\item $\zeta_C(1;\sigma\psi)\zeta_C(1)=\zeta_C(2;\sigma\psi)-\zeta_C\left(\begin{matrix}\sigma & \psi\\
1 & 1\end{matrix}\right)-\zeta_C\left(\begin{matrix}\psi & \sigma\\
1 & 1\end{matrix}\right)+\zeta_C\left(\begin{matrix}\boldsymbol{1} & \sigma\psi\\
1 & 1\end{matrix}\right)+\zeta_C\left(\begin{matrix} \sigma\psi& \boldsymbol{1}\\
1 & 1\end{matrix}\right)$.
\end{enumerate}
\end{Theorem}
We observe that the formula (1) can be found in Thakur's \cite[Theorem 5.10.13]{THA}. Further, due to 
the symmetry of the roles of $t_1$ and $t_2$, the formulas (2) and (3) are equivalent.
Moreover, the formula (4) is in fact well known (see Perkins' \cite{PER} for more general 
formulas of this type). However, we will give a proof of this in the spirit of multiple zeta values.
The formulas (2) is, on the other side, new, as far as we can see. 

\section{Twisted power sums}

We need a few tools in order to obtain our formulas; more precisely, we have to improve our
skill in computing twisted powers sums. For this purpose, we are going to use the tools introduced
in the recent preprint of Perkins and the author \cite{PEL&PER}; we are going to use for a 
while the notations of this reference. Let $s$ be an integer $\geq0$. We set, for an integer $d\geq 0$:
$$S_d(n;s)=\sum_{a\in A^+(d)}\frac{a(t_1)\cdots a(t_s)}{a^n}\in K[\underline{t}_s].$$ We are 
thus considering a special case of Definition \ref{defpowersums}.
We also set, for $d\geq 0$, $F_0(n;s)=0$ and 
$$F_d(n;s)=\sum_{i=0}^{d-1}S_d(n;s)\in K[\underline{t}_s],$$
so that $$\lim_{d\rightarrow\infty}F_d(n;s)=\zeta_C(n;s):=\prod_P\left(1-\frac{P(t_1)\cdots P(t_s)}{P^n}\right)^{-1}\in\TT_s^\times,$$ where the product runs over the irreducible polynomials of $A^+$ (in general, all along this note,
empty sums are by convention equal to zero). We are using the notation of \cite{PEL&PER}.
In particular, if $\sigma$ is the semi-character $\chi_{t_1}\cdots\chi_{t_s}$ (of Dirichlet type),
then, the comparison between the notations of ibid. and those of the present note are:
$$S_d(n;s)=S_d(n;\sigma),\quad \zeta_C(n;s)=\zeta_C(n;\sigma).$$

It is easy to show that, if $0\leq s'<s$, 
$S_d(n;s')$ is the coefficient of $(t_{s'+1}\cdots t_s)^{d}$ in $F_{d+1}(n;s)$. 
We define inductively $l_0=1$ and $l_i=(\theta-\theta^{q^i})l_{i-1}$, and we set $l_{-n}=0$ for $n>0$.
We denote by $b_i(Y)$ the product $(Y-\theta)\cdots(Y-\theta^{q^{i-1}})\in A[Y]$ (for an indeterminate $Y$) if $i>0$ and we set
$b_0(Y)=1$. We also write $m=\lfloor\frac{s-1}{q-1}\rfloor$ (the brackets denote the integer part so that $m$ is the biggest integer $\leq\frac{s-1}{q-1}$). We set
$$\Pi_{s,d}=\frac{b_{d-m}(t_1)\cdots b_{d-m}(t_s)}{l_{d-1}}\in K[\underline{t}_s],\quad d\geq \max\{1,m\}.$$

Now,
we quote \cite[Theorem 1]{PEL&PER}: 
\begin{Theorem}\label{Theorem2}
For all integers $s\geq 1$, such that $s\equiv1\pmod{q-1}$, there exists a non-zero rational fraction $\mathbb{H}_{s}\in K(Y,\underline{t}_s)$
such that, for all $d\geq m$, the following identity holds:
$$F_d(1;s)=\Pi_{s,d}\mathbb{H}_s|_{Y=\theta^{q^{d-m}}}.$$
If $s=1$, we have the explicit formula $$\mathbb{H}_1=\frac{1}{t_1-\theta}.$$
Further, if $s=1+m(q-1)$ for an integer $m>0$, then
the fraction $\mathbb{H}_s$ is a 
polynomial of $A[Y,\underline{t}_s]$ with the following properties:
\begin{enumerate}
\item For all $i$, $\deg_{t_i}(\mathbb{H}_s)=m-1$,
\item $\deg_{Y}(\mathbb{H}_s)=\frac{q^m - 1}{q-1} - m$.
\end{enumerate}
The polynomial $\mathbb{H}_s$ is uniquely determined by these properties.
\end{Theorem} 
We apply this Theorem to compute explicitly the twisted power sums associated to 
the data we have chosen. For this, it suffices to choose $s=q$. In this case $m=1$ and the 
Polynomial $\mathbb{H}_q$ of Theorem \ref{Theorem2} has degree $0$ in $Y$ as
well as in $t_1,\ldots,t_q$. From \cite[\S 2.6]{PEL&PER} we deduce that 
$\mathbb{H}_q=1$ (the same result is given as an example in Florent Demeslay's thesis \cite{DEM}). In particular, to compute most of the twisted power sums associated to 
our data it suffices to analyze the polynomials
\begin{equation}\label{Fdq}
F_{d+1}(1;q)=\frac{b_{d}(t_1)\cdots b_{d}(t_q)}{l_{d}}.\end{equation}
The coefficients of $(t_3\cdots t_q)^d$, $(t_2\cdots t_q)^d$ and $(t_1\cdots t_q)^d$ in $F_{d+1}(1;q)$ are easily computed, and we get, for all $d\geq 0$ (note that these are well known formulas; see \cite{PER}):
\begin{eqnarray}
S_d(1;0)&=&\frac{1}{l_d},\label{e1}\\
S_d(1;1)&=&\frac{b_d(t_1)}{l_d},\label{e2}\\
S_d(1;2)&=&\frac{b_d(t_1)b_d(t_2)}{l_d}.\label{e3}
\end{eqnarray}
To compute $S_d(2;0),S_d(2,1),S_d(2,2)$ (\footnote{In the notations of \cite{PEL&PER};
in our note, we should write 
$S_d(2;\boldsymbol{1}),S_d(2;\sigma),S_d(2;\sigma\psi)$.}) we observe that,
replacing $\theta$ with $\theta^q$ in 
(\ref{Fdq}): $$F_{d+1}(q;q)=\frac{b_{d+1}(t_1)\cdots b_{d+1}(t_q)}{l_d^q(t_1-\theta)\cdots(t_q-\theta)}.$$
Note that the former is a polynomial in $\underline{t}_s$, written as a reducible fraction.
We get that $$F_{d+1}(2;2)=\left.\frac{b_{d+1}(t_1)\cdots b_{d+1}(t_q)}{l_d^q(t_1-\theta)\cdots(t_q-\theta)}\right|_{t_3=\cdots=t_q=\theta}=\frac{b_{d+1}(t_1)b_{d+1}(t_2)}{l_d^2(t_1-\theta)(t_2-\theta)}.$$
Calculating the coefficients of $t_2^d$ and $(t_1t_2)^d$, and 
subtracting $F_{d+1}(2;2)-F_d(2;2)$, we easily obtain the formulas, valid for $d\geq 0$ (the first one is well known):
\begin{eqnarray}
S_d(2;0)&=&\frac{1}{l_d^2},\nonumber     \\
S_d(2;1)&=&\frac{b_d(t_1)}{(t_1-\theta)l_d^2}(t_1-\theta^{q^d}),  \label{f2}  \\
S_d(2;2)&=& \frac{b_d(t_1)b_d(t_2)}{(t_1-\theta)(t_2-\theta)l_d^2} (t_1t_2-\theta^{q^d}(t_1+t_2)+2\theta^{1+q^d}-\theta^2)\nonumber\\
&=&\frac{b_d(t_1)b_d(t_2)}{(t_1-\theta)(t_2-\theta)l_d^2}[(t_1-\theta)(t_2-\theta)+(t_1-\theta)(\theta-\theta^{q^d})+(t_2-\theta)(\theta-\theta^{q^d})]\label{f3}.
\end{eqnarray}
We will also need the next Lemma which is also well known, where $\tau:A[t]\rightarrow A[t]$ is the $\FF_q[t]$-linear 
endomorphism such that $\tau(\theta)=\theta^q$.
\begin{Lemma}\label{lemmatau} We have the following formula, which holds in $A[t]$.
$$\tau(b_d(t))=l_d\sum_{i=0}^d\frac{b_i(t)}{l_i},\quad d\geq0.$$
\end{Lemma}
\begin{proof} We recall the proof for convenience of the reader.
We proceed by induction over $d$. If $d=0$, the formula is obvious. If $d>0$,
it suffices to show that $$b_{d+1}(t)=(t-\theta)l_d\sum_{i=0}^dl_i^{-1}b_i(t).$$
Now, we compute easily, by using the induction hypothesis:
\begin{eqnarray*}
b_{d+2}(t)&=&(t-\theta+\theta-\theta^{q^d})b_{d+1}\\
&=&(\theta-\theta^{q^d})b_{d+1}+(t-\theta)b_{d+1}\\
&=&(t-\theta)l_d(\theta-\theta^{q^d})\left(\sum_{i=0}^dl_i^{-1}b_i(t)+l_{d+1}^{-1}b_{d+1}(t)\right)\\
&=&(t-\theta)l_{d+1}\sum_{i=0}^{d+1}l_i^{-1}b_i(t).
\end{eqnarray*}
\end{proof}
\section{Proof of Theorem \ref{formulas}}

\begin{proof}[Proof of Theorem \ref{formulas}, (2).]
We compute, by using (\ref{e1}) and (\ref{e2}):
$$S_d(1;\sigma)S_d(1;\boldsymbol{1})=\frac{b_d(t_1)}{l_d^2}.$$

On the other hand, we have seen in (\ref{f2}) that 
$$S_d(2;\sigma)=\frac{b_d(t_1)}{l_d^2}\frac{t_1-\theta^{q^d}}{t_1-\theta}=\frac{\tau(b_d(t_1))}{l_d^2},$$
where $\tau(b_d(t_1))$ has the obvious meaning. We compute (the third identity follows from Lemma
\ref{lemmatau}):
\begin{eqnarray*}
S_d\left(\begin{matrix}\boldsymbol{1} & \sigma\\
1 & 1\end{matrix}\right)&=&S_d(1;\boldsymbol{1})\sum_{i=0}^{d-1}S_i(1;\sigma)\\
&=&\frac{1}{l_d}\sum_{i=0}^{d-1}l_i^{-1}b_i(t_1)\\
&=&\frac{\tau(b_{d-1}(t_1))}{l_dl_{d-1}}\\
&=&\frac{b_d(t_1)}{l_d^2}\frac{\theta-\theta^{q^d}}{t_1-\theta}.
\end{eqnarray*}
Combining the above formulas, we see that
\begin{equation}\label{first}
S_d(1;\sigma)S_d(1;\boldsymbol{1})=S_d(2;\sigma)-S_d\left(\begin{matrix}\boldsymbol{1} & \sigma\\
1 & 1\end{matrix}\right).
\end{equation}
With the obvious meaning of some new notations introduced below, we deduce:
\begin{eqnarray*}
F_d(1;\sigma)F_d(1;\boldsymbol{1})&=&\sum_{i=0}^{d-1}S_i(1;\sigma)\sum_{j=0}^{d-1}S_j(1;\boldsymbol{1})\\
&=&F_d\left(\begin{matrix}\boldsymbol{1} & \sigma\\
1 & 1\end{matrix}\right)+F_d\left(\begin{matrix}\sigma& \boldsymbol{1} \\
1 & 1\end{matrix}\right)+ \sum_{i=0}^{d-1}S_d(1;\sigma)S_d(1;\boldsymbol{1})\\
&=& F_d\left(\begin{matrix}\boldsymbol{1} & \sigma\\
1 & 1\end{matrix}\right)+F_d\left(\begin{matrix}\sigma& \boldsymbol{1} \\
1 & 1\end{matrix}\right)+\sum_{i=0}^{d-1}\left(S_i(2;\sigma)-S_i\left(\begin{matrix}\boldsymbol{1} & \sigma\\
1 & 1\end{matrix}\right)\right)\\
&=&F_d\left(\begin{matrix}\boldsymbol{1} & \sigma\\
1 & 1\end{matrix}\right)+F_d\left(\begin{matrix}\sigma& \boldsymbol{1} \\
1 & 1\end{matrix}\right)-F_d\left(\begin{matrix}\boldsymbol{1} & \sigma\\
1 & 1\end{matrix}\right)+F_d(2;\sigma)\\
&=&F_d\left(\begin{matrix}\sigma& \boldsymbol{1} \\
1 & 1\end{matrix}\right)+F_d(2;\sigma),\end{eqnarray*}
where we have used (\ref{first}) in the third equality.
We rewrite the resulting formula:
\begin{equation}\label{Fsfirst}
F_d(1;\sigma)F_d(1;\boldsymbol{1})=F_d\left(\begin{matrix}\sigma& \boldsymbol{1} \\
1 & 1\end{matrix}\right)+F_d(2;\sigma).
\end{equation}
Taking the limit $d\rightarrow\infty$ in (\ref{Fsfirst}) we obtain the required multiple zeta identity.
\end{proof}
The above also implies the formula (3) of Theorem \ref{formulas} by interchanging $t_1$ and 
$t_2$.
To continue, we notice the next Lemma.
\begin{Lemma}\label{alemma}
We have that
$$S_d(2;\sigma\psi)=\frac{b_d(t_1)b_d(t_s)}{(t_1-\theta)(t_2-\theta)l_d^2}+S_d\left(\begin{matrix}\psi & \sigma\\
1 & 1\end{matrix}\right)+S_d\left(\begin{matrix}\sigma & \psi\\
1 & 1\end{matrix}\right).$$
\end{Lemma}
\begin{proof}
We compute:
\begin{eqnarray*}
S_d\left(\begin{matrix}\sigma & \psi\\
1 & 1\end{matrix}\right)&=&S_d(1;\sigma)\sum_{i=0}^{d-1}S_i(1;\psi)\\
&=&\frac{b_d(t_1)}{l_d}\sum_{i=0}^{d-1}\frac{b_i(t_2)}{l_i}\\
&=&\frac{b_d(t_1)b_d(t_s)}{(t_1-\theta)(t_2-\theta)l_d^2}[(t_1-\theta)(\theta-\theta^{q^d})],
\end{eqnarray*}
in virtue of (\ref{e2}) and Lemma \ref{lemmatau}. Similarly,
we have
$$S_d\left(\begin{matrix}\psi & \sigma\\
1 & 1\end{matrix}\right)=\frac{b_d(t_1)b_d(t_s)}{(t_1-\theta)(t_2-\theta)l_d^2}[(t_2-\theta)(\theta-\theta^{q^d})].$$
The lemma follows from (\ref{f3}).
\end{proof}

\begin{proof}[Proof of Theorem \ref{formulas}, (4).] As we have already mentioned, this is 
a well known formula but we want to provide a new proof.
We note, by (\ref{e2}), that 
$$S_d(1;\sigma)S_d(1;\psi)=\frac{b_d(t_1)b_d(t_2)}{(t_1-\theta)(t_2-\theta)l_d^2}[(t_1-\theta)(t_2-\theta)].$$
Hence, Lemma \ref{alemma} implies the formula
\begin{equation}\label{formulabis}
S_d(1;\sigma)S_d(1;\psi)=S_d(2;\sigma\psi)-S_d\left(\begin{matrix}\psi & \sigma\\
1 & 1\end{matrix}\right)-S_d\left(\begin{matrix}\sigma & \psi\\
1 & 1\end{matrix}\right).
\end{equation}
We deduce:
\begin{eqnarray*}
F_d(1;\sigma)F_d(1;\psi)&=&F_d\left(\begin{matrix}\psi & \sigma\\
1 & 1\end{matrix}\right)+F_d\left(\begin{matrix}\sigma & \psi\\
1 & 1\end{matrix}\right)+\sum_{i=0}^{d-1}S_d(1;\sigma)S_d(1;\psi)\\
&=&F_d\left(\begin{matrix}\psi & \sigma\\
1 & 1\end{matrix}\right)+F_d\left(\begin{matrix}\sigma & \psi\\
1 & 1\end{matrix}\right)+\sum_{i=0}^{d-1}\left(S_d(2;\sigma\psi)-S_d\left(\begin{matrix}\psi & \sigma\\
1 & 1\end{matrix}\right)-S_d\left(\begin{matrix}\sigma & \psi\\
1 & 1\end{matrix}\right)\right)\\
&=&F_d\left(\begin{matrix}\psi & \sigma\\
1 & 1\end{matrix}\right)-F_d\left(\begin{matrix}\psi & \sigma\\
1 & 1\end{matrix}\right)+F_d\left(\begin{matrix}\sigma & \psi\\
1 & 1\end{matrix}\right)-F_d\left(\begin{matrix}\sigma & \psi\\
1 & 1\end{matrix}\right)+F_d(2;\sigma\psi)\\
&=&F_d(2;\sigma\psi),
\end{eqnarray*}
where we have applied the formula (\ref{formulabis}) in the third equality.
The formula of the theorem follows by letting $d\rightarrow\infty$.
\end{proof}

\begin{proof}[Proof of Theorem \ref{formulas}, (5).]
The identities (\ref{e1}) and (\ref{e3}) imply that 
$$S_d(1;\boldsymbol{1})S_d(1;\sigma\psi)=\frac{b_d(t_1)b_d(t_2)}{(t_1-\theta)(t_2-\theta)l_d^2}[(t_1-\theta)(t_2-\theta)].$$
Lemma \ref{alemma} then implies that also:
\begin{equation}\label{formulater}
S_d(1;\boldsymbol{1})S_d(1;\sigma\psi)=S_d(2;\sigma\psi)-S_d\left(\begin{matrix}\psi & \sigma\\
1 & 1\end{matrix}\right)-S_d\left(\begin{matrix}\sigma & \psi\\
1 & 1\end{matrix}\right).
\end{equation}
We deduce:
\begin{eqnarray*}
F_d(1;\boldsymbol{1})F_d(1;\sigma\psi)&=&F_d\left(\begin{matrix}\boldsymbol{1} & \sigma\psi\\
1 & 1\end{matrix}\right)+F_d\left(\begin{matrix}\sigma\psi & \boldsymbol{1}\\
1 & 1\end{matrix}\right)+\sum_{i=0}^{d-1}S_d(1;\sigma)S_d(1;\psi)\\
&=&F_d\left(\begin{matrix}\boldsymbol{1} & \sigma\psi\\
1 & 1\end{matrix}\right)+F_d\left(\begin{matrix}\sigma\psi & \boldsymbol{1}\\
1 & 1\end{matrix}\right)+\sum_{i=0}^{d-1}\left(S_d(2;\sigma\psi)-S_d\left(\begin{matrix}\psi & \sigma\\
1 & 1\end{matrix}\right)-S_d\left(\begin{matrix}\sigma & \psi\\
1 & 1\end{matrix}\right)\right)\\
&=& F_d(2;\sigma\psi)+F_d\left(\begin{matrix}\boldsymbol{1} & \sigma\psi\\
1 & 1\end{matrix}\right)+F_d\left(\begin{matrix}\sigma\psi & \boldsymbol{1}\\
1 & 1\end{matrix}\right) - F_d\left(\begin{matrix}\psi & \sigma\\
1 & 1\end{matrix}\right) - F_d\left(\begin{matrix}\sigma & \psi\\
1 & 1\end{matrix}\right),\end{eqnarray*}
so we have reached the formula
\begin{equation}\label{lastone}
F_d(1;\boldsymbol{1})F_d(1;\sigma\psi)=F_d(2;\sigma\psi)+F_d\left(\begin{matrix}\boldsymbol{1} & \sigma\psi\\
1 & 1\end{matrix}\right)+F_d\left(\begin{matrix}\sigma\psi & \boldsymbol{1}\\
1 & 1\end{matrix}\right) - F_d\left(\begin{matrix}\psi & \sigma\\
1 & 1\end{matrix}\right) - F_d\left(\begin{matrix}\sigma & \psi\\
1 & 1\end{matrix}\right).
\end{equation}
Letting $d$ tend to $\infty$ in (\ref{lastone}) concludes the proof.
\end{proof}
\begin{Remark}\label{trivialremark}{\em In fact, it is trivial that
$$\zeta_C(1)\zeta_C(1;\sigma\psi)-\zeta_C(1;\sigma)\zeta_C(1,\psi)=
\zeta_d\left(\begin{matrix}\boldsymbol{1} & \sigma\psi\\
1 & 1\end{matrix}\right)+\zeta_d\left(\begin{matrix}\sigma\psi & \boldsymbol{1}\\
1 & 1\end{matrix}\right) - \zeta_d\left(\begin{matrix}\psi & \sigma\\
1 & 1\end{matrix}\right) - \zeta_d\left(\begin{matrix}\sigma & \psi\\
1 & 1\end{matrix}\right),$$ and that the corresponding identity for the
sums $F_d$ holds as well.
Indeed, setting $\alpha_i=S_i(1;\boldsymbol{1})=l_i^{-1}$,
$\beta_i=S_i(1;\sigma)=b_i(t_1)l_i^{-1}$, $\gamma_i=S_i(1;\psi)=b_i(t_2)l_i^{-1}$
and $\delta_i=S_i(1;\sigma\psi)=b_i(t_1)b_i(t_2)l_i^{-1}$, we
see that
\begin{eqnarray*}
\lefteqn{\sum_{i\geq 0}\alpha_i\sum_{j\geq 0}\delta_j-\sum_{i\geq 0}\beta_i\sum_{j\geq 0}\gamma_j=}\\
&=&\sum_{i>j\geq 0}\alpha_i\delta_j+\sum_{j>i\geq 0}\alpha_i\delta_j-\sum_{i>j\geq 0}\beta_i\gamma_j-\sum_{j>i\geq 0}\beta_i\gamma_j+\\
& &+\sum_{i\geq 0}\alpha_i\delta_i-\sum_{i\geq 0}\beta_i\gamma_i.
\end{eqnarray*}
But, of course, $\alpha_i\delta_i=\beta_i\gamma_i$ for all $i$, from which the identity follows.
Up to this simple trick, the identity (5) of Theorem \ref{formulas} can be considered as equivalent to
the identity (4) of the same result.}\end{Remark}
\section{Some consequences}

In virtue of Proposition \ref{entireness} or by direct verification, the identities of Theorem \ref{formulas} involve 
entire functions in two variables $t_1,t_2$. Hence, specializing the variables,
we are able to recover identities in $\CC_\infty$ or in some intermediate 
Tate algebra. We are going to show several results
arising from the formula (2). 
We recall the formula (2) of Theorem \ref{formulas} for convenience:
\begin{equation}\label{eq2}
\zeta_C(1;\sigma)\zeta_C(1)=\zeta_C\left(\begin{matrix}\sigma & \boldsymbol{1}\\
1 & 1\end{matrix}\right)+\zeta_C(2;\sigma).\end{equation}
First of all, we can replace $t_1$ by $\theta$, but this does not give 
anything interesting; we mention it only to show how the substitution works. We recall the following formula
that the author proved in \cite{PEL}:
\begin{equation}\label{identityannals}
\zeta_C(1;\sigma)=\frac{\widetilde{\pi}}{(\theta-t_1)\omega(t_1)},\end{equation}
where 
$$\omega(t_1)=(-\theta)^{\frac{1}{q-1}}\prod_{i\geq 0}\left(1-\frac{t_1}{\theta^{q^i}}\right)^{-1}\in\TT_1^\times,$$ is the {\em Anderson-Thakur function} (note that
$\Omega(t_1)=\frac{1}{(t_1-\theta)\omega(t_1)}$ is an entire function; see \cite{ANG&PEL2}, containing
a recent overview on the known properties of this function).
The function $\omega(t_1)$ having a simple pole of residue $-\widetilde{\pi}$ at $t_1=\theta$,
we see that $\zeta_C(1;\sigma)|_{t_1=\theta}=1$. Now, it is easy to see that
\begin{equation}\label{eq3}
\zeta_C\left(\begin{matrix}\boldsymbol{1} & \sigma\\
1 & 1\end{matrix}\right)=\sum_{d\geq 0}S_d(1;\sigma)\sum_{i=0}^{d-1}l_i^{-1}=
\sum_{d\geq 0}l_d^{-1}b_d(t_1)\sum_{i=0}^{d-1}l_i^{-1}\end{equation}
vanishes at $t_1=\theta$. Further, $$\zeta_C(2;\sigma)=\sum_{d\geq 0}S_d(2;\sigma)$$ takes the 
value $\zeta_C(1)$ at $t_1=\theta$. Hence, with this evaluation, we only get 
the tautological identity $\zeta_C(1)=\zeta_C(1)$.

\subsection{A family of multiple zeta identitities}
We can also evaluate this identity at $t_1=\theta^{q^{-k}}$ with $k>0$ and raise the obtained identity 
to the power $q^k$. Working out the intermediate details, the reader will easily recover the following 
sum shuffle formula:
$$\zeta_C(q^k)\zeta_C(q^k-1)=\zeta_C(2q^k-1)+\zeta_C(q^k-1,q^k),\quad k\geq 1.$$

\subsection{Evaluation at trivial zeros}
Now, we evaluate the second identity of Theorem \ref{formulas} at $t_1$ equal to a trivial zero of the function $\zeta_C(1;\sigma)$ which, as it appears from 
the computation of the poles of the gamma factor of (\ref{identityannals}), means that we replace
$t_1$ with $\theta^{q^d}$ with $d>0$. This implies the following result.
\begin{Theorem}
The following formula holds
\begin{equation}\label{formulaBG}
\operatorname{BG}_{q^d-2}=-\sum_{d\geq i>j\geq 0}\frac{b_i(\theta^{q^d})}{l_il_j},\quad d\geq 1.
\end{equation}
\end{Theorem} 
\begin{proof} In fact, to make things more transparent, we make a step back to the identity  (\ref{Fsfirst}) 
that we rewrite as
$$F_k(2;\sigma)=F_k(1;\sigma)F_k(1;\boldsymbol{1})-F_k\left(\begin{matrix}\sigma & \boldsymbol{1}\\
1 & 1\end{matrix}\right),\quad k\geq 0.$$ By Proposition \ref{entireness}, we see that 
all the sequences $F_k(\cdots)$ involved tend, for $k\rightarrow\infty$, to entire functions
of the variable $t_1$. However, we already know from \cite[Proposition 6]{ANG&PEL}
that $F_k(2;\sigma)$ and $F_k(1;\sigma)F_k(1;\boldsymbol{1})$ tend to entire functions, 
so we immediately obtain that $F_k\left(\begin{smallmatrix}\sigma & \boldsymbol{1}\\
1 & 1\end{smallmatrix}\right)$ tends to an entire function as $k\rightarrow\infty$ without using Proposition \ref{entireness} (in fact, this can be also seen directly).
Replacing $t_1=\theta^{q^d}$ with $d>0$ yields the value zero for the limit $\lim_{k\rightarrow\infty}F_k(1;\sigma)F_k(1;\boldsymbol{1})=\zeta_C(1;\sigma)\zeta_C(1)$ evaluated at $t_1=\theta^{q^d}$. Indeed, after (\ref{identityannals}), $\theta^{q^d}$ is a trivial zero of $\zeta_C(1;\sigma)$.
Further, we see that $$\zeta_C(2;\sigma)|_{t_1=\theta^{q^d}}=\lim_{k\rightarrow\infty}F_k(2;\sigma)|_{t_1=\theta^{q^d}}=\sum_{k\geq 0}\sum_{a\in A^+(k)}a^{q^d-2}=\operatorname{BG}_{q^d-2}\in A.$$
Moreover, by (\ref{eq3}) and evaluating at $t_1=\theta^{q^d}$,
$$
\zeta_C\left(\begin{matrix}\sigma & \boldsymbol{1} \\ 1 & 1\end{matrix}\right)_{t_1=\theta^{q^d}}=
\lim_{k\rightarrow\infty}F_k\left(\begin{matrix}\sigma & \boldsymbol{1}\\
1 & 1\end{matrix}\right)_{t_1=\theta^{q^d}}=
\sum_{i>j\geq 0}\frac{b_i(\theta^{q^d})}{l_il_j}=\sum_{d\geq i>j\geq 0}\frac{b_i(\theta^{q^d})}{l_il_j},
$$
because $b_i(\theta^{q^d})$ vanishes for all $i>d$.
\end{proof}

One nice aspect of the formula (\ref{formulaBG}) is that it is easy to reduce it
modulo an irreducible polynomial of $A$ of degree $d$.
The following family of congruences is an immediate consequence of our result, and was first observed
by Thakur in \cite{THA0}, and Angl\`es and Ould Douh in \cite{ANG&DOU}:
\begin{Corollary}\label{TAOD} For all $P$ an irreducible polynomial of $A^+(d)$, we have (recall that $|P|=q^d$):
$$\operatorname{BG}_{|P|-2}\equiv\sum_{i=0}^{d-1}\frac{1}{l_i}\equiv F_d(1;\boldsymbol{1})\pmod{P}.$$
\end{Corollary}
\begin{proof} For all $i,j$ with $d\geq i>j\geq 0$, the fraction $b_i(\theta^{q^d})(l_il_j)^{-1}$ is
$P$-integral for any $P$ irreducible of degree $d$.
If $i<d$, we have that $b_i(\theta^{q^d})l_i^{-1}\equiv0\pmod{P}$, because 
$b_i(\theta^{q^d})$ is divisible, in $A$, by $\theta^{q^d}-\theta$. Further,
$b_d(\theta^{q^d})/l_d=-1$, from which the congruence follows.
\end{proof}
\begin{Remark}{\em 
The reader can do similar computations
with other formulas; more results will appear elsewhere.
Observe, however, that manipulating in the same way the  formula (5) of Theorem \ref{formulas}
returns relatively less identities. The reason seems to be that this formula is 
trivially equivalent to the formula (4), as it follows from Remark \ref{trivialremark}.
At least, we deduce, specializing $t_1=\theta^{q^{-k}}$ and $t_2=\theta^{q^{-h}}$, the following strange sum shuffle identity, valid for $h,k\geq 0$ with $h+k>0$:
\begin{eqnarray*}
\lefteqn{\zeta_C(1)^{q^k}\zeta_C(q^{k+h}-q^h-1)=}\\
&=&\zeta_C(2q^{h+k}-q^h-1)+\zeta_C(q^{k+h},q^{k+h}-q^h-1)+\zeta_C(q^{k+h}-q^h-1,q^{k+h})-\\
 & &-\zeta_C(q^{k+h}-q^h,q^{k+h}-1)-\zeta_C(q^{k+h}-1,q^{k+h}-q^h).\end{eqnarray*}
The formula (2) of Theorem \ref{formulas} and especially the formula
of Theorem \ref{formulaBG} can be seen as some kind of analogue 
of Euler's formula $\zeta(3)=\zeta(2,1)$ (for classical Euler multiple zeta values).}
\end{Remark}

\subsubsection{A degree computation}
In contrast with the universal formulas obtained in \cite{PEL&PER} for the sums 
$F_d(n;s)$ in the case $s\equiv n\pmod{q-1}$, there seems to be no such a formula for $\operatorname{BG}_{q^d-2}$, for $d\geq 1$. At least, we have a ``universal formula" for its degree, and this can
be deduced from (\ref{formulaBG}) as we are going to see in the next result, where 
it is supposed, again (as we did until now), that $q>2$.

\begin{Theorem}\label{exactdegree}
We have $$\deg_\theta(\operatorname{BG}_{q^d-2})=(d-1)q^d-\frac{2q(q^{d-1}-1)}{q-1}.$$
\end{Theorem}
This result should be compared with more classical degree computations by Wan, Diaz-Vargas, Poonen, 
Sheats, as well as B\"ockle's \cite[Theorem 1.2]{BOE} where the interested reader can find all the 
necessary references to the work of these authors on this topic. 

Theorem \ref{exactdegree} seems to be new.
Thomas \cite[Theorem 2]{THO} already proved an explicit formula to compute,
not only the degree of $\operatorname{BG}_{q^d-2}$ for $d>0$ but also 
the degree of $\operatorname{BG}_{n}$ for any $n>0$ with $q-1\nmid n$ in case $q$ is a prime number (\footnote{He obtained
a more general result in this direction, also involving ``first derivatives" of the Goss zeta function of $A$ at
its ``trivial zeroes," the negative integers divisible by $q-1$.}). It follows from Thomas' Theorem 1 and Corollary 1  ibid.
that if $q-1\nmid n$, and $q$ is a prime, then $\operatorname{BG}_{n}=1$ if and only if $\ell_q(n)<q$,
where $\ell_q(n)$ is the sum of the digits of the base-$q$ expansion of $n$.
However, Thomas formula contains an iterative process and for this reason, the identity of Theorem 
\ref{exactdegree} is not immediately recognizable in it, and even if it was, it would have been valid only for $q$ a prime number. Bruno Angl\`es has communicated to the author a simple proof of 
Theorem \ref{exactdegree} in the case of $q=p>2$ a prime number by using Sheats' method. Also, Dinesh Thakur has pointed out to the author that this result, for general $q$,
can be more simply deduced from an application of his {\em duality formula} \cite[Theorem 2, (5)]{THA3}.

\begin{Remark}{\em 
Note that for all $d\geq 1$, $q^d-2$ is a {\em dual magic number} in the sense of \cite[\S 5.7]{GOS?} (see also \cite{GOS}). In this paper,
Goss points out a result of Thomas which exhibits a computation of the degree of 
$\operatorname{BG}_{n}$ when $q-1\nmid n$ and when $n$ is a {\em magic number} (\cite[\S 8.22]{GOS}), in terms of 
the the degree of the Carlitz factorial.
Our computation
involves certain dual magic numbers which are not magic numbers, and this could also be a new instance of
the conjectural functional equation for the Goss zeta function associated to the algebra $A$.}
\end{Remark}

More results of the type of Theorem \ref{exactdegree} can be obtained from more general consequences of the sum shuffle 
relations for our multi-zeta values in the Tate algebras, but they will be described in
another work (the present paper can be considered as a first of more general results that will appear elsewhere). Before proving the Theorem, we need some notation and Lemmas.
For commodity, we set
$$\alpha_i=\frac{b_i(\theta^{q^d})}{l_i},\quad \beta_j=l_j^{-1},$$
so that the formula (\ref{formulaBG}) rewrites as
$$\operatorname{BG}_{q^d-2}=-\sum_{d\geq i>j\geq 0}\alpha_i\beta_j.$$
Then, we have
\begin{equation}\label{formulaalphabeta}
\delta_{i,j}:=\deg_\theta(\alpha_i\beta_j)=iq^d-\sum_{n=1}^iq^n-\sum_{m=1}^{j}q^m.
\end{equation} We recall the convention that an empty sum is zero. Moreover, the 
degree of $0$ in $\theta$ is set to be $-\infty$.
We have the following Lemma.
\begin{Lemma} Assuming that $d\geq i>j\geq 0$, $d\geq i'>j'\geq 0$, we have that
$$\delta_{i,j}=\deg_\theta(\alpha_i\beta_j)=\deg_\theta(\alpha_{i'}\beta_{j'})=\delta_{i',j'}$$
if and only if the following cases occur.
\begin{enumerate}
\item $i=i',j=j'$,
\item $i=d$, $i'=d-1$ and $j=j'$,
\item $i'=d,i=d-1$ and $j=j'$.
\end{enumerate}
\end{Lemma}
\begin{proof}
For symmetry of the roles of $i,i'$, we can assume that $i\geq i'$.
First of all, if $i=i'$ we have that
$$\delta_{i,j}-\delta_{i',j'}=\sum_{m'=1}^{j'}q^{m'}-\sum_{m=1}^{j}q^m,$$
which equals zero if and only if $j=j'$; if $i=i'$, $\delta_{i,j}=\delta_{i',j'}$ if and only if $j=j'$. Now, let us suppose that $i>i'$. From 
(\ref{formulaalphabeta}) we deduce that
$$\delta_{i,j}-\delta_{i',j'}=(i-i')q^d-\psi_{i,i',j,j'},$$
where $$\psi_{i,i',j,j'}=\sum_{n=1}^iq^{n}-\sum_{n'=1}^{i'}q^{n'}+\sum_{m=1}^{j}q^{m}-\sum_{m'=1}^{j'}q^{m'}\in\ZZ.$$
Since $i>i'>j'$ and $i>j$, we have that $\psi_{i,i',j,j'}\geq 0$ and we can find 
integers $c_r\in\{0,1,2\}$, unique, such that 
$\psi_{i,i',j,j'}=\sum_{r=0}^ic_rq^r,$ so we see that $$0\leq \psi_{i,i',j,j'}< q^{i+1}$$
(recall that $q>2$). If $i<d$, we see that
$$\delta_{i,j}-\delta_{i',j'}=(i-i')q^d-\psi_{i,i',j,j'}> q^d-q^{i+1}\geq0$$ so that $\delta_{i,j}\neq\delta_{i',j'}$
in this case. It remains to study the case in which 
$i=d$. In this case, $j'<i'\leq d-1$ and we can write:
$$\delta_{i,j}-\delta_{i',j'}=(d-i'-1)q^d-\rho_{i',j,j'},$$
where 
$$\rho_{i',j,j'}=\sum_{n=i'+1}^{d-1}q^{n}+\sum_{m=1}^{j}q^{m}-\sum_{m'=1}^{j'}q^{m'}\in\ZZ.$$
This number is obviously $\geq 0$ and the following estimate holds 
$$0\leq\rho_{i',j,j'}<q^{d}.$$ If $i'\leq d-2$, we thus have that $\delta_{d,j}>\delta_{i',j'}$
for any choice of $j<d$ and $j'<i'$. If $i'=d-1$, we have 
$$\delta_{i,j}-\delta_{i',j'}=\sum_{m'=1}^{j'}q^{m'}-\sum_{m=1}^{j}q^{m}$$
which equals zero if and only if $j=j'$.
\end{proof}
In view of the above Lemma, to compute the degree of $\operatorname{BG}_{q^d-2}$, we rearrange the sum (\ref{formulaBG}) in the following way:
\begin{eqnarray*}
\operatorname{BG}_{q^d-2}&=&-\overbrace{\alpha_d\beta_{d-1}}^U-\underbrace{(\alpha_d+\alpha_{d-1})\sum_{j=0}^{d-2}\beta_j}_V-\underbrace{\sum_{i=0}^{d-2}\sum_{j=0}^{i-1}\alpha_i\beta_j}_W\\
&=:&-(U+V+W).
\end{eqnarray*}
\begin{Lemma}\label{degreesUVW}
We have:
\begin{enumerate}
\item $\deg_\theta(U)=(d-1)q^d-2(q+\cdots+q^{d-1})$,
\item $\deg_\theta(V)=(d-2)q^d-(q+\cdots+q^{d-2})$,
\item $\deg_\theta(W)=(d-2)q^d-(q+\cdots+q^{d-2})$.
\end{enumerate}
\end{Lemma}
\begin{proof}
We compute the degree of $U$:
$$\deg_\theta(U)=dq^d-\sum_{n=1}^dq^n-\sum_{m=1}^{d-1}q^m=(d-1)q^d-2\sum_{n=1}^{d-1}q^n.$$
To compute the degree of $V$, we observe:
\begin{eqnarray*}
\alpha_d+\alpha_{d-1}&=&\frac{b_d(\theta^{q^d})}{l_d}-\frac{b_{d-1}(\theta^{q^d})}{l_{d-1}}\\
&=&\frac{b_d(\theta^{q^d})-(\theta^{q^d}-\theta)b_{d-1}(\theta^{q^d})}{l_d}\\
&=&\frac{b_{d-1}(\theta^{q^d})[\theta^{q^d}-\theta^{q^{d-1}}-\theta^{q^d}+\theta]}{l_d}\\
&=&\frac{(\theta-\theta^{q^{d-1}})b_{d-1}(\theta^{q^d})}{l_d}.
\end{eqnarray*}
Hence, 
\begin{eqnarray*}
\deg_\theta(V)&=&\deg_\theta(\alpha_d+\alpha_{d-1})+\deg_\theta\left(\sum_{j=0}^{d-2}\beta_j\right)\\
&=&\deg_\theta(\alpha_d+\alpha_{d-1})\\
&=&q^{d-1}+(d-1)q^d-\sum_{n=1}^dq^n\\
&=&(d-2)q^d-\sum_{n=1}^{d-2}q^n.
\end{eqnarray*}
To compute the degree of $W$, we first notice, by Lemma \ref{formulaalphabeta}, 
that all the terms involved in the double sum have different degrees. The term with the
largest degree is the one corresponding to $i=d-2$ and $j=0$, which is equal to 
$\alpha_{d-2}$, and which has the expected degree.
\end{proof}

\begin{proof}[Proof of Theorem \ref{exactdegree}]
By Lemma \ref{degreesUVW} and by the assumption $q>2$, we have $$\deg_\theta(U)>
\deg_\theta(V),\deg_\theta(W),$$
and 
$$\deg_\theta(\operatorname{BG}_{q^d-2})=\deg_\theta\left(\frac{b_d(\theta^{q^d})}{l_dl_{d-1}}\right)=
(d-1)q^d-\frac{2q(q^{d-1}-1)}{q-1}.$$
\end{proof}
\begin{Remark}{\em 
Angl\`es and Ould Douh have proved, in \cite{ANG&DOU}, that
there exist infinitely many irreducible elements of $A^+$ such that
$\operatorname{BG}_{|P|-2}\not\equiv0\pmod{P}$ (recall that $|P|=q^d$ in (\ref{formulaBG})
and that $q>2$ all along the present note). 
As a consequence we see, by the fact that $$F_{\deg_\theta(P)}(1;\sigma)\not\equiv0\pmod{P}$$ 
for all irreducible $P$ of $A^+$ (easily checked), that the right-hand sides of (\ref{Fsfirst}) and (\ref{lastone}) determine non-zero elements 
of the ring $\mathcal{A}_s$ defined in \cite{PEL&PER}. This result is an easy consequence 
of their formula  that we have re-obtained in our Corollary \ref{TAOD}. 

Let us recall the elegant proof of this property in \cite{ANG&DOU}. Since
$\operatorname{BG}_{q^d-2}\equiv\sum_{i=0}^{d-1}l_i^{-1}\pmod{P}$ (for $P$ irreducible of
degree $d$), we have $\operatorname{BG}_{q^d-2}\equiv0\pmod{P}$ if and only if
$P$ divides the polynomial $$V(d)=l_{d-1}\sum_{i=0}^{d-1}l_i^{-1}\in A$$ which has degree
$\sum_{n=1}^{d-1}q^n=\frac{q^d-q}{q-1}$, so that we have at most $$\frac{q^d-q}{d(q-1)}$$
monic irreducible polynomials $P$ of degree $d$ dividing $V(d)$.
Now, the number of monic irreducible polynomials $P$ of degree $d$ is equal to the 
 the necklace
polynomial (where $\mu$ designates Moebius' function) $$M_d(q)=\frac{1}{d}\sum_{l\mid d}\mu(l)q^{\frac{d}{l}},$$ which is known to have an asymptotic behavior, as $d\rightarrow\infty$, which is of a strictly bigger magnitude than that of
the above fraction if $q>2$.
For example, if $d=p'$ is a prime
number, the necklace polynomial $M_{p'}(q)$ equals $\frac{q^{p'}-q}{p'}$
and we have 
$$\frac{q^{p'}-q}{p'}>\frac{q^{p'}-q}{p'(q-1)},$$ because $q>2$.

The formula (\ref{formulaBG}) does not seem to immediately imply the result of Angl\`es and Ould Douh
(without using the intermediate congruence with the polynomial $V(d)$), 
but we have not tried to rearrange the terms of the sum completely.}\end{Remark}

\section{Looking for more relations}

We gave above some hints of a variant of the shuffle product for the 
multiple zeta values:
$$\zeta_C\left(\begin{matrix}\sigma_1 & \sigma_2 & \cdots & \sigma_r\\
n_1 & n_2 & \cdots & n_r\end{matrix}\right)$$ in the simplest non-trivial cases (weight $2$).
We shall complete our note by suggesting some other tools to develop, in order to 
compute other kinds of relations. 

We denote by $K\{\tau\}$ the 
skew polynomial ring of finite sums $\sum_ic_i\tau^i$, with $c_i\in K$, with the 
non-commutative product uniquely determined by the rule $\tau c=c^q\tau$
for $c\in K$. Additionally, let $t$ be a variable (we can set $t=t_1$ to get compatibility with the first part of the note). We have an isomorphism of $K$-vector spaces:
$$K[t]\xrightarrow{\eta} K\{\tau\}$$
defined by $\eta(t^i)=C_{\theta^i}=(\theta+\tau)^i$ for $i>0$ and $\eta(1)=1$.
Here $C_\theta=\theta+\tau$ is the multiplication by $\theta$ of the Carlitz module.
The inverse of this isomorphism sends $1$ to $1$ and, for $j>0$, $\tau^j$ to $b_j(t)$, where we recall that
$$b_j(t)=(t-\theta)\cdots(t-\theta^{q^{j-1}}).$$
To check that $\eta$ is an isomorphism, one uses the {\em evaluation} 
at the Anderson-Thakur function. The evaluation $f(\omega)$ of an element 
$f=f_0+f_1\tau+\cdots+f_r\tau^r\in K\{\tau\}$ at $\omega$ is by definition
the expression $(f_0+f_1b_1+\cdots+f_rb_r)\omega$. It is easy to see that $C_a(\omega)=a(t)\omega$, so that, for all $f(t)\in K[t]$, we have 
$$\eta(f)(\omega)=f(t).$$

This isomorphism $\eta$ is useful to construct certain identities for finite sums. We recall, as a first example, the formula
(\ref{e2}):
$$\label{verysimple}S_d(1;1)=S_d(1;\sigma)=\sum_{a\in A^+(d)}a^{-1}a(t)=\frac{b_d(t)}{l_d},\quad d\geq 0.$$

It is easy to show that $\eta(a(t))=C_a\in K\{\tau\}$. Therefore, the isomorphism 
$\eta$ yields the identity:
$$\eta(S_d(1;1))=\sum_{a\in A^+(d)}a^{-1}C_a=l_d^{-1}\tau^d,\quad d\geq 0.$$

This picture can be generalized. We can use variables $t_1,\ldots,t_s$, indeterminates
$\tau_1,\ldots,\tau_s$, the rings $K[t_1,\ldots,t_s]$ (commutative) and 
$K\{\tau_1,\ldots,\tau_s\}$ (non commutative, with multiplication rules: $\tau_i\tau_j=\tau_j\tau_i$ and $\tau_ic=c^q\tau_i$ for $c\in K$), and the isomorphism
$$K[t_1,\ldots,t_s]\xrightarrow{\eta}K\{\tau_1,\ldots,\tau_s\}$$
uniquely defined by $\eta(t_i^j)=(\theta+\tau_i)^j$ (we write $\eta$ instead of the more precise expression $\eta_s$ we should have used,
to simplify our notations).
Then, any time we can show a formula for power sums in $K[t_1,\ldots,t_s]$, we obtain a similar formula in the ring $K\{\tau_1,\ldots,\tau_s\}$.

Florent Demeslay proved, in his Thesis \cite{DEM}, the following result.
\begin{Theorem} Assume that $s>0$.
There exists a rational fraction $Q_{k,s}\in K(t_1,\ldots,t_s)(Y)$ such that
$$S_d(k;s)=\frac{b_d(t_1)\cdots b_d(t_s)}{l_d}Q_{k,s}(\theta^{q^{d-m}}),\quad d\geq 0$$
where $m=\max\{0,\lfloor\frac{s-1}{q-1}\rfloor\}$.
\end{Theorem}
The case $s=0$ (no variables) was already known to Anderson and Thakur \cite{AND&THA}.
We would like to apply this Theorem for $s>0$ to produce identities in the non-commutative
indeterminates $\tau_1,\ldots,\tau_s$ by means of the isomorphism $\eta$. 
For example, if $k=1$ and $s=1$, we are reduced to the formula (\ref{e2})
with $Q_{1,1}=1$. However, there is no reason to 
expect that $Q_{k,s}$ is a polynomial and in fact, in general, this is false.
For example, it is easy to check that $Q_{q,1}=\frac{t-Y}{t-\theta}$, which is not a polynomial.
It is of course possible to compute the rational fractions $Q_{k,s}$ by using the polynomials
$\mathbb{H}_s$ of Theorem \ref{Theorem2}, but even with that in mind, 
we cannot escape this problem.

A partial solution is given by Lemma \ref{lemmatau}.
We now denote by $\boldsymbol{\tau}$ (we do not want to mix it up with $\tau$ which is now an indeterminate!) the $\FF_q[t]$-algebra 
endomorphism of $K[t]$ defined by $\boldsymbol{\tau}(c)=c^q$. Then, Lemma \ref{lemmatau}
and induction imply the following result.
\begin{Proposition}\label{proposition4}
For all $n>0$ and $d\geq 0$, the following formula holds:
$$\boldsymbol{\tau}^n(b_d(t))=l_d^{q^{n-1}}\sum_{d\geq i_1\geq i_2\geq \cdots\geq i_{n-1}\geq i_n\geq 0}l_{i_1}^{q^{n-2}-q^{n-1}}l_{i_2}^{q^{n-3}-q^{n-2}}\cdots l_{i_{n-1}}^{1-q}l_{i_n}^{-1}b_{i_n}(t).$$
In particular, for all $n>0$ and $d\geq 0$:
$$S_d(q^n;1)=l_d^{q^{n-1}-q^n}\sum_{d\geq i_1\geq i_2\geq \cdots\geq i_{n-1}\geq i_n\geq 0}l_{i_1}^{q^{n-2}-q^{n-1}}l_{i_2}^{q^{n-3}-q^{n-2}}\cdots l_{i_{n-1}}^{1-q}l_{i_n}^{-1}b_{i_n}(t).$$
\end{Proposition}

Although the rational fraction $Q_{q,1}$ and more generally the fractions $Q_{q^n,1}$ are certainly not polynomials, the above formulas in $K[t]$ can be transferred to identities in the ring $K\{\tau\}$. We obtain,
by applying the map $\eta$:
\begin{Corollary}\label{noncommide}
For all $d\geq 0$, 
$$\mathfrak{S}_d(q^n;1):=\sum_{a\in A^+(d)}a^{-q^n}C_a=l_d^{q^{n-1}-q^n}\sum_{d\geq i_1\geq i_2\geq \cdots\geq i_{n-1}\geq i_n\geq 0}l_{i_1}^{q^{n-2}-q^{n-1}}l_{i_2}^{q^{n-3}-q^{n-2}}\cdots l_{i_{n-1}}^{1-q}l_{i_n}^{-1}\tau^n.$$
\end{Corollary}
Let $\sigma_1,\ldots,\sigma_r$ be 
semi-characters, let $n_1,\ldots,n_r$ be integers, and $d$ a non-negative integer. 
We set, for convenience:
$$S_d^\star\left(\begin{matrix}\sigma_1 & \sigma_2 & \cdots & \sigma_r\\
n_1 & n_2 & \cdots & n_r\end{matrix}\right)=
S_{d}(n_1;\sigma_1)\sum_{d\geq  i_2\geq  \cdots\geq  i_r\geq 0}
S_{i_2}(n_2;\sigma_2)\cdots S_{i_r}(n_r;\sigma_r)\in \FF_q^{ac}\otimes_{\FF_q}K(\underline{t}_s)$$
(we have introduced non-strict inequalities in the sum).
Further, we set:
$$\zeta_C^\star\left(\begin{matrix}\sigma_1 & \sigma_2 & \cdots & \sigma_r\\
n_1 & n_2 & \cdots & n_r\end{matrix}\right):=\sum_{d\geq 0}S_d^*\left(\begin{matrix}\sigma_1 & \sigma_2 & \cdots & \sigma_r\\
n_1 & n_2 & \cdots & n_r\end{matrix}\right)\in \widehat{K_\infty\otimes_{\FF_q}\FF_q^{ac}\otimes_{\FF_q}\boldsymbol{F}_s}.$$
We observe that $S_d(j;\boldsymbol{1})=l_d^{-j}$ if $j=kq^l$ with $l\geq 0$ and $k=1,\ldots,q-1$.
Hence, the second identity of Proposition \ref{proposition4} can be rewritten, with $\sigma=\chi_t$, in the following
way:
$$S_d(q^n;\sigma)=S_d^\star\left(\begin{matrix}\boldsymbol{1} & \boldsymbol{1} & \cdots & \boldsymbol{1} &\sigma \\
q^{n-1}(q-1) & q^{n-2}(q-1) & \cdots & q-1 & 1 \end{matrix}\right).$$
Summing over $d=0,1,\ldots$, we obtain the 
formula:
$$\zeta_C(q^n;\sigma)=\zeta_C^\star\left(\begin{matrix}\boldsymbol{1} & \boldsymbol{1} & \cdots & \boldsymbol{1} &\sigma \\
q^{n-1}(q-1) & q^{n-2}(q-1) & \cdots & q-1 & 1 \end{matrix}\right).$$We observe that the evaluation at $t=\theta^{q^k}$ gives:
$$\underbrace{\ldots,\theta^{q^{n-1}}}_{\text{special values}\neq0},\underbrace{\theta^{q^n}}_{\text{value one}},\underbrace{\theta^{q^{n+1}},\theta^{q^{n+2}},\ldots}_{\text{trivial zeroes}}.$$
Evaluating e.g. at $t=\theta$ returns us the following identity, with the obvious 
meaning of the second sum:
$$\zeta_C(q^n-1)=\zeta_C^\star(\underbrace{q^{n-1}(q-1),q^{n-2}(q-1),\ldots,q-1}_{n\text{ terms}}).$$
We can rewrite the identity of our Corollary \ref{noncommide} as follows:
\begin{eqnarray*}
\lefteqn{\mathfrak{S}_d(q^n;1)=\sum_{a\in A^+(d)}a^{-q^n}C_a=}\\
&=&S_d(q^{n-1}(q-1))\sum_{d\geq i_1\geq i_2\geq \cdots\geq i_{n-1}\geq i_n\geq 0}S_{i_1}(q^{n-2}(q-1))S_{i_2}(q^{n-3}(q-1))\cdots S_{i_{n-1}}(q-1)S_{i_n}(1)\tau^n,
\end{eqnarray*}
with $S_d(n):=S_d(n;\boldsymbol{1})$.
If $f=f_0+f_1\tau+\cdots+f_r\tau^r\in K\{\tau\}$, the {\em evaluation at one} $f(1)$
of $f$ is the element $f_0+f_1+\cdots+f_r\in K$. It is easy to see that the series
$\sum_{d\geq 0}\sum_{a\in A^+(d)}a^{-q^n}C_a(1)$ converges in $K_\infty$.
We obtain the formula:
\begin{equation}
\sum_{d=0}^\infty\mathfrak{S}_d(q^n;1)(1)=\zeta^\star_C(\underbrace{q^{n-1}(q-1),q^{n-2}(q-1),\ldots,q-1}_{n\text{ terms}},1).
\end{equation}
These formulas can be easily related to Thakur's multiple zeta values $\zeta_C$ (without the $\star$ mark),
by means of simple manipulations. We illustrate this in the case $n=1$. We observe that
$$\sum_{d=0}^\infty\mathfrak{S}_d(q;1)(1)=\zeta^\star_C(q-1,1)=\zeta_C(q-1,1)+\sum_{i\geq 0}S_{i}(q-1;\boldsymbol{1})S_i(1,\boldsymbol{1}).
$$
Now, since $S_i(q-1;\boldsymbol{1})=l_i^{1-q}$ and $S_i(1;\boldsymbol{1})=l_i$, we get 
$$\sum_{i\geq 0}S_{i}(q-1;\boldsymbol{1})S_i(1;\boldsymbol{1})=\sum_{i\geq 0}l_i^{-q}=\log_C(1)^q,$$ where
$\log_C(z)=\sum_{i\geq 0}l_i^{-1}z^{q^i}$ is the {\em Carlitz logarithm} of $z\in\CC_\infty$,
well defined for $|z|<q^{q/(q-1)}$ (we recall that $|\cdot|$ is the unique norm
of $\CC_\infty$ such that $|\theta|=q$) and in particular, well defined at $z=1$. It is plain
that $\log_C(1)=\zeta_C(1)$, an identity which was, essentially, first noticed by Carlitz.
Thus we have that
$$\zeta^\star_C(q-1,1)=\zeta_C(q-1,1)+\zeta_C(1)^q.$$ 

The shuffle product of $\zeta_C(s_1)$ and $\zeta_C(s_2)$ yields, for $s_1,s_2\in\NN^+$
such that $s_1+s_2\leq q$ (see Thakur, \cite[Theorem 1]{THA1}), the simple formula:
$$\zeta_C(1)\zeta_C(q-1)=\zeta_C(q)+\zeta_C(q-1,1)+\zeta_C(1,q-1)=\zeta_C(1)^q+\zeta_C(q-1,1)+\zeta_C(1,q-1).$$ This means that 
$$\sum_{d=0}^\infty\mathfrak{S}_d(q;1)(1)=\zeta^\star_C(q-1,1)=\zeta_C(1)\zeta_C(q-1)-\zeta_C(1,q-1).$$
We do not know how to evaluate the sum $\sum_{d=0}^\infty\mathfrak{S}_d(q;1)(1)$
(and more generally, similar sums we do not want to introduce in this paper) directly, and it would be nice to develop a technique to do so independently of the shuffle product, in order to re-obtain the shuffle 
product formula.
Note also that Thakur demonstrated the formula (see \cite[Theorem 5]{THA1}):
$$\zeta_C(m,m(q-1))=\frac{\zeta_C(mq)}{(\theta-\theta^q)^m},\quad m=1,\ldots,q-1.$$
Hence, we compute easily, with $m=1$:
\begin{eqnarray*}
\zeta_C^\star(q-1,1)&=&\zeta_C(q-1,1)+\zeta_C(1)^q\\
&=&\zeta_C(q-1)\zeta_C(1)-\zeta_C(1,q-1)\\
&=&\zeta_C(1)\left(\zeta_C(q-1)-\frac{\zeta_C(1)^{q-1}}{\theta-\theta^q}\right).
\end{eqnarray*}

\begin{Remark}
{\em In the examples we have studied above, the semi-characters are all of Dirichlet type but 
for no reason this should be considered as a necessary condition for the existence of shuffle-like formulas.
For example, if $\nu:A^+\rightarrow\FF_q[t]$ is the semi-character which which associates $a\in A^+$
to $t^{\deg_\theta(a)}$ (this is not of Dirichlet type), then the following formula holds in the 
Tate algebra $\TT$, as the 
reader can easily check:
$$\zeta_C(1;\nu)\zeta_C(1;\boldsymbol{1})=\zeta_C(2;\nu)+\zeta_C\left(\begin{matrix}\nu & \boldsymbol{1} \\
1 & 1 \end{matrix}\right)+\zeta_C\left(\begin{matrix}\boldsymbol{1} & \nu \\
1 & 1 \end{matrix}\right).$$ Evaluating at $t=1$ we deduce the formula (1) of Theorem \ref{formulas}.}
\end{Remark}

\subsubsection*{Acknowledgements} The author is thankful to Bruno Angl\`es for fruitful discussions, to David Goss for having suggested 
some references, including \cite{THO} and \cite{GOS?}, and to Dinesh Thakur for 
having pointed out several inaccuracies in a previous version of the present note.

\end{document}